\documentclass[11pt,english]{article}
\usepackage[T1]{fontenc}
\usepackage[latin9]{inputenc}
\usepackage{geometry}
\usepackage{color}
\usepackage{refstyle}
\usepackage{mathtools}
\usepackage{amsmath}
\usepackage{amsthm}
\usepackage{amssymb}
\usepackage{stackrel}
\PassOptionsToPackage{normalem}{ulem}
\usepackage{ulem}

\makeatletter

\RS@ifundefined{subsecref}
  {\newref{subsec}{name = \RSsectxt}}
  {}
\RS@ifundefined{thmref}
  {\def\RSthmtxt{theorem~}\newref{thm}{name = \RSthmtxt}}
  {}
\RS@ifundefined{lemref}
  {\def\RSlemtxt{lemma~}\newref{lem}{name = \RSlemtxt}}
  {}

\theoremstyle{plain}
\newtheorem{thm}{\protect\theoremname}[section]
\newtheorem*{thmnn}{Theorem}
\theoremstyle{definition}
\newtheorem{defn}[thm]{\protect\definitionname}
\theoremstyle{plain}
\newtheorem{prop}[thm]{\protect\propositionname}
\ifx\proof\undefined
\newenvironment{proof}[1][\protect\proofname]{\par
	\normalfont\topsep6\p@\@plus6\p@\relax
	\trivlist
	\itemindent\parindent
	\item[\hskip\labelsep\scshape #1]\ignorespaces
}{%
	\endtrivlist\@endpefalse
}
\providecommand{\proofname}{Proof}
\fi
\theoremstyle{plain}
\newtheorem{lem}[thm]{\protect\lemmaname}
\theoremstyle{definition}

\theoremstyle{remark}
\newtheorem{rem}[thm]{\protect\remarkname}
\theoremstyle{plain}
\newtheorem{cor}[thm]{\protect\corollaryname}

\usepackage{ytableau}
\renewenvironment{bmatrix}{\begin{ytableau}}{\end{ytableau}}

\makeatother

\usepackage{babel}
\providecommand{\corollaryname}{Corollary}
\providecommand{\definitionname}{Definition}
\providecommand{\examplename}{Example}
\providecommand{\lemmaname}{Lemma}
\providecommand{\propositionname}{Proposition}
\providecommand{\remarkname}{Remark}
\providecommand{\theoremname}{Theorem}

\begin{document}

\title{Polynomial Expressions for the Dimensions of the Representations of Symmetric Groups and Restricted Standard Young Tableaux}

\author{Avichai Cohen and Shaul Zemel}

\maketitle

\section*{Introduction}

The representations of the symmetric group $S_{n}$ are parameterized by the partitions of $n$. The paper \cite{[R]} considers the representations of $S_{n}$ that have minimal dimensions, and discovers that as $n$ grows, these representations correspond to partitions with a very long first row (up to transposition). This means that we fix a (typically small) number $k$ and a partition $\lambda$ of $k$, and consider, for large $n$, the partition $(n-k,\lambda)$ obtained by adding to $\lambda$ a single long first row to get a partition of $n$.

It is observed that the dimension of the representation obtained in this way, with fixed $k$ and $\lambda$ and letting $n$ grow, is a polynomial of degree $k$ in $n$. Moreover, this polynomial is integer-valued, and the author proves that when expanding this polynomial in the basis consisting of differences of successive binomial polynomials, the coefficients are non-negative. This is used for proving that the dimensions of the representations obtained from small $k$ are indeed minimal for large $n$, and determine, in some cases, the threshold for the value of $n$ starting from which each such representation is indeed of minimal dimension.

We remark that the results of \cite{[R]} are for complex representations, and the paper \cite{[J]} contains some extensions of these results to representations over other fields.

\smallskip

The space of integer-valued polynomials has, however, a ``more natural'' basis, consisting of the binomial polynomials. In this paper we investigate, for every partition $\lambda$ of a non-negative integer $k$, the expansion of the aforementioned polynomial in $n$ in the basis of binomial polynomials. We prove that rather than $k$ polynomials, only the $\ell(\lambda)+1$ maximal binomial polynomials of degree at most $k$ show up, where $\ell(\lambda)$ is the number of non-zero summands in $\lambda$ (this is unlike the expansions from \cite{[R]}, which usually require all of the basis elements, except for one which is proved to never show up---see Theorem D of that reference). The coefficients of the binomial polynomials that do show up are alternating in sign, and we prove that they count restricted standard Young tableaux of shape $\lambda$.

To state the result explicitly, we need some notation. Using the standard notation from \cite{[S]} and others for the dimensions of the representations of symmetric group, we consider $f^{(n-k,\lambda)}$ as a polynomial in $n$. The statement that only the maximal $\ell(\lambda)+1$ basis elements participate in the expansion means that we can write \[f^{(n-k,\lambda)}=\sum_{h=0}^{\ell(\lambda)}(-1)^{h}a_{\lambda,h}\binom{n}{k-h},\] with the signs showing up expressing the alternating property.

For any partition $\lambda$ of $k$, we denote by $\operatorname{SYT}_{h}(\lambda)$ the set of those standard Young tableaux of shape $\lambda$ in which the numbers between 1 and $h$ are posited at the beginning of the first $h$ rows. Our main result, which is Theorem \ref{coeffSYT} below, is as follows.
\begin{thmnn}
For every partition $\lambda$ and every $0 \leq h\leq\ell(\lambda)$, the coefficient $a_{\lambda,h}$ equals the cardinality of the set $\operatorname{SYT}_{h}(\lambda)$.
\end{thmnn}
In view of this result, the fact that there are no such restricted tableaux when $h>\ell(\lambda)$ corresponds to the vanishing of the coefficients associated with such values of $h$. Note that for $h=0$ and $h=1$ there are no restrictions, and the case $h=2$ yields a decomposition of the full set of standard Young tableaux of shape $\lambda$, which implies that for symmetric partitions of numbers that are larger than 1 this set, which always has an even cardinality, has a ``natural choice'' of a representative of any pair of transpose tableaux (see Proposition \ref{SYT2} and Corollary \ref{symev} below).

\smallskip

The restriction that we posed on the Young tableaux involved the numbers between 1 to $h$, and forced them to be in increasing rows of the tableau. One can pose a similar restriction, but on the numbers between $1+\alpha$ and $h+\alpha$ for any $0\leq\alpha \leq k-h$. We denote the set of standard Young tableaux of shape $\lambda$ satisfying the restriction associated with $h$ and $\alpha$ by $\operatorname{SYT}_{h,\alpha}(\lambda)$, and in fact, the first comparison that we proved for the coefficients $a_{\lambda,h}$ was with the cardinality of $\operatorname{SYT}_{h,k-h}(\lambda)$.

However, Y. Roichman and R. Adin shared with us that the cardinalities of these sets are the same for every value of $\alpha$, and the proof is easier for the case $\alpha=0$ as stated above. This independence can be obtained by considering the results of the paper \cite{[SW]}, combined with Sch\"{u}tzenberger's jeu de taquin (see Section 3.7 of \cite{[S]} for background on the latter). It is also a consequence of Theorem 7.2 of \cite{[AR]}, and another way to obtain it is by combining Corollary 4.2 of \cite{[M]} with the fact that the sum of fundamental quasi-symmetric functions that are associated with the descent set of a standard Young tableau yields the corresponding Schur symmetric function (see, e.g., \cite{[G]}, and the related paper \cite{[GR]}).

A recurring theme in Algebraic Combinatorics is to find, whenever two sets of objects have the same cardinality, an explicit and natural bijection between them. Our second result is the construction of such a bijection for consecutive values of $\alpha$.
\begin{thmnn}
Given a partition $\lambda$ of an integer $k$, an integer $0 \leq h\leq\ell(\lambda)$, and another integer $1\leq\alpha \leq k-h$, there exist explicit inverse maps $\varphi_{\lambda,h,\alpha}:\operatorname{SYT}_{h,\alpha}(\lambda)\to\operatorname{SYT}_{h,\alpha-1}(\lambda)$ and $\phi_{\lambda,h,\alpha}:\operatorname{SYT}_{h,\alpha-1}(\lambda)\to\operatorname{SYT}_{h,\alpha}(\lambda)$, which are the identity on elements of the intersection, and which commute with adding and removing boxes as in the branching rule.
\end{thmnn}
The most general formula is given in Remark \ref{genmaps} below.

\smallskip

The paper is divided into 4 sections. Section \ref{BinPol} gives the basic notions involving integer-valued polynomials. Section \ref{PolofPart} introduces partitions, the associated polynomials, and gives the first properties of their expansion in the binomial polynomials. Section \ref{ResTab} investigates the restricted standard Young tableaux, and proves the result relating them to the coefficients in the expansion from above. Finally, Section \ref{Bijections} constructs the bijections between the sets of restricted standard Young tableaux with consecutive values of $\alpha$, proves their properties, and concludes with an example.

\smallskip

We are thankful to N. Linial, G. Kalai, B. Rhoades, and B. Sagan for some advice regarding this paper. We are very much indebted to Y. Roichman and R. Adin for their interest in the question, for letting us know about the independence result of the parameter $\alpha$, for making us acquainted with several related reference, and for numerous fascinating discussions on this topic.

\section{The Binomial Basis for Rational Polynomials \label{BinPol}}

In this section we present the binomial polynomials as a basis for the integer-valued polynomials over $\mathbb{Z}$, with some of their properties. This material is classical and well-known, with many generalizations to rings other than $\mathbb{Z}$ (see the old references \cite{[P]} and \cite{[O]}, the more comprehensive treatise \cite{[CC1]}, some advances in \cite{[Z]} and \cite{[CC2]}, and an application to Number Theory in \cite{[dSI]}), and we include it mainly to set up the notation that will be used later in the paper.

\smallskip

For $h\in\mathbb{N}$, the \emph{binomial polynomial} is defined by \[p_{h}(x):=\frac{x\cdot(x-1)\cdot(x-2)\cdot\ldots\cdot\big(x-(h-1)\big)}{h!}=\frac{\prod_{i=0}^{h-1}(x-i)}{h!},\] including the case $p_{0}:=1$. It is clear that $p_{h}\in\mathbb{Q}[x]$, and that it has degree $h$. It follows immediately from these properties that for every $k\geq0$, the set $\{p_{h}\}_{h=0}^{k}$ is a basis for the space of polynomials from $\mathbb{Q}[x]$ having degree at most $k$ (and thus $\{p_{h}\}_{h=0}^{\infty}$ is a basis for $\mathbb{Q}[x]$).

The polynomials $p_{h}$ are \emph{integer-valued polynomials}, in the following sense.
\begin{lem}
For any $h\in\mathbb{N}$ and $n\in\mathbb{Z}$ we have $p_{k}(n)\in\mathbb{Z}$. \label{intval}
\end{lem}

\begin{proof}
When $n\geq0$ then $p_{h}(n)$ is the binomial coefficient $\binom{n}{k}$, which is known to be an integer. When $n<0$ we note that $p_{h}$ satisfies the functional equation $p_{h}(x)=(-1)^{h}p_{h}(h-1-x)$, which reduces this case to the previous one. This proves the lemma.
\end{proof}

The basis property implies that every polynomial $p(x)$ of degree at most $k$ in $\mathbb{Q}[x]$ can be written uniquely as $\sum_{h=0}^{k}a_{h}p_{h}(x)$, and Lemma \ref{intval} implies that if $a_{h}\in\mathbb{Z}$ for every $h$ then $p$ is also integer-valued. The converse, and even a slight strengthening of it, is also true.
\begin{prop}
If $p(x)=\sum_{h=0}^{k}a_{h}p_{h}(x)$ satisfies the property that $p(n)\in\mathbb{Z}$ for every integer $n\gg0$, then $a_{h}\in\mathbb{Z}$ for all $0 \leq h \leq k$. \label{coeffsinZ}
\end{prop}

\begin{proof}
The proof is by induction on $k$. If $k=0$ then $p(x)$ is a constant, which is thus an integer by substituting a large enough integer $n$.

Assume now that the result holds for $k-1$, and take a polynomial $p(x)=\sum_{h=0}^{k}a_{h}p_{h}(x)$ with our property. It follows that $p(n+1)-p(n)$ is also an integer for any $n\gg0$, and we note that for a positive integer $n$ we have $p_{h}(n+1)-p_{h}(n)=\binom{n+1}{h}-\binom{n}{h}=\binom{n}{h-1}=p_{h-1}(n)$ (via Pascal's identity). It follows that $p_{h}(x+1)-p_{h}(x)=p_{h-1}(x)$ as polynomials (as both sides coincide on infinitely many values of the variable), and we get \[p(x+1)-p(x)=\sum_{h=0}^{k}a_{h}\big(p_{h}(x+1)-p_{h}(x)\big)=\sum_{h=1}^{k}a_{h}p_{h-1}(x)=\sum_{h=0}^{k-1}a_{h+1}p_{h}(x).\] But then this is the expansion of a polynomial of degree $k-1$ with our property, so that the induction hypothesis yields, after replacing $h+1$ by $h$ back again, that $a_{h}\in\mathbb{Z}$ for all $h\geq1$. But then $p(x)$ is the sum of $\sum_{h=1}^{k}a_{h}p_{h}(x)$, which is integer-valued by Lemma \ref{intval}, and the constant $a_{0}$. By substituting a large enough integer again, we deduce that $a_{0}\in\mathbb{Z}$ as well. This completes the proof of the proposition.
\end{proof}

We will also need the following property of the coefficients in our expansions.
\begin{lem}
Write a polynomial $p(x)\in\mathbb{Q}[x]$, of degree at most $k$, as $\sum_{h=0}^{k}a_{h}p_{h}(x)$ again. Then for every $0 \leq t \leq k$ we have $a_{h}=0$ for every $0 \leq h \leq t$ if and only if $p(m)=0$ for every $0 \leq m \leq t$. \label{vancoeffs}
\end{lem}

\begin{proof}
Assume that $a_{h}=0$ for all $0 \leq h \leq t$, so that $p(x)=\sum_{h=t+1}^{k}a_{h}p_{h}(x)$. Then for any integer $n\geq0$ we get $p(n)=\sum_{h=t+1}^{k}a_{h}\binom{n}{h}$, and if $n \leq t<h$ for every $h$ showing up in this sum, then $p(n)=0$ as desired.

For the other direction we argue by induction on $t$. When $t=0$ we substitute $x=0$ in the polynomial, and as in the previous paragraph we get $p(0)=\sum_{h=0}^{k}a_{h}\binom{0}{h}=a_{0}$. Hence if $p(0)=0$ then $a_{0}=0$ as asserted.

Assume now that the assertion holds for $t-1$, and that $p(m)=0$ for every $0 \leq m \leq t$. Thus $p(m)=0$ for every $0 \leq n \leq t-1$, and the induction hypothesis implies that $a_{h}=0$ for all $0 \leq h \leq t-1$, meaning that $p(x)=\sum_{h=t}^{k}a_{h}p_{h}(x)$. We now substitute $x=t$, and get $0=p(t)=\sum_{h=t}^{k}a_{h}\binom{t}{h}=a_{t}$ as above, yielding the additional required vanishing. This proves the lemma.
\end{proof}

\section{The Polynomial Associated with a Partition \label{PolofPart}}

We begin by recalling a few definitions. A \emph{partition} $\lambda$ is a finite sequence $\lambda:=(\lambda_{1},\ldots,\lambda_{\ell})$ of non-increasing positive integers. The number $\ell$ of integers showing up in the partition is called the \emph{length} of $\lambda$, and denoted by $\ell(\lambda)$. We say that $\lambda$ is a \emph{partition of $k$}, denoted as $\lambda \vdash k$, if $\sum_{i=1}^{\ell}\lambda_{i}=k$.

Partitions are often described using their Young diagrams, also knows as Ferrers diagrams (or just diagrams). The \emph{diagram of shape $\lambda$} is an array of boxes having $\ell(\lambda)$ left-justified rows, such that the $i$'th row is of length $\lambda_{i}$. We will be using the English notation, meaning that $i=1$ is the top row and $i=\ell(\lambda)$ is the bottom one. We will identify $\lambda$ with its diagram, and given positive integers $i$ and $j$, the statement $(i,j)\in\lambda$ means that the box in row $i$ and column $j$ in the diagram of shape $\lambda$.

Recall that given a partition $\lambda\vdash k$, a \emph{Young tableau of shape $\lambda$} is a diagram of shape $\lambda$, in which with the boxes are filled bijectively by the integers between 1 and $k$. A Young tableau of shape $\lambda$ is \emph{standard} if the numbers in each row and each column are put in an increasing order. We set
\[\operatorname{SYT}(\lambda):=\{T\;|\;T\text{ is a standard Young tableau of shape }\lambda\},\quad\mathrm{and}\quad f^{\lambda}:=|\operatorname{SYT}(\lambda)|.\]

For example, if $\lambda=(2,1)\vdash3$ than all the Young tableaux of shape $\lambda$ are
\[\begin{bmatrix}1 & 2 \\ 3\end{bmatrix},\ \begin{bmatrix}2 & 1 \\ 3\end{bmatrix},\ \begin{bmatrix}1 & 3 \\ 2\end{bmatrix},\ \begin{bmatrix}3 & 1 \\ 2\end{bmatrix},\ \begin{bmatrix}2 & 3 \\ 1\end{bmatrix},\mathrm{\ and\ }\begin{bmatrix}3 & 2 \\ 1\end{bmatrix},\]
among which the first and the third are the standard ones, so that $f^{\lambda}=2$.

\smallskip

The numbers $f^{\lambda}$ have a lot of interesting properties, of which we will use two. The first one is a direct evaluation using the \emph{hook formula} of Frame, Robinson, and Thrall, which we now describe. If $\lambda$ is any partition, and $v=(i,j)$ is a box in the diagram of shape $\lambda$, the we define it's \emph{hook} and \emph{hook length} to be
\[H_{v}:==H_{i,j}:=\{(i,j')\in\lambda\;|\;j' \geq j\}\cup\{(i',j)\in\lambda\;|\;i'\geq i\}\quad\mathrm{and}\quad h_{v}=h_{i,j}=|H_{i,j}|\] respectively (both depending also on $\lambda$, despite the notation not including it).

The result is now as follows.
\begin{thm}[Hook Formula]
For any $\lambda \vdash k$ we have the equality \[f^{\lambda}=\frac{k!}{\prod_{v\in\lambda}h_{v}}.\] \label{hookform}
\end{thm}
This result is stated and proved as, e.g., Theorem 3.10.2 of \cite{[S]}.

Take now a partition $\lambda \vdash k$, and consider a large enough integer $n$ (in fact, $n\geq\lambda_{1}$ will suffice). We write $(n-k,\lambda)$ for the partition of $n$ that is obtained by putting a row of length $n-k$ in front of $\lambda$, namely
$(n-k,\lambda):=(n-k,\lambda_{1},\lambda_{2},\ldots,\lambda_{\ell(\lambda)}) \vdash n$. Our object of interest is $f^{(n-k,\lambda)}$ as a function of $n$.

Its first property, the first part of which also shows up as Theorem A of \cite{[R]}, is as follows.
\begin{prop}
If $\lambda \vdash k$, then there is a polynomial $p_{\lambda}\in\mathbb{Q}[x]$, of degree $k$, such that the equality $f^{(n-k,\lambda)}=p_{\lambda}(n)$ holds for any $n\gg0$. The polynomial $p_{\lambda}$ is divisible by $\prod_{i=0}^{k-\ell(\lambda)-1}(n-i)$. \label{fpol}
\end{prop}

\begin{proof}
The hook formula from Theorem \ref{hookform} yields the equality \[f^{(n-k,\lambda)}=\frac{n!}{\prod_{i,j}h_{i,j}},\] where the $h_{i,j}$'s are the hook lengths of the diagram of $(n-k,\lambda)$. Note that if $i\geq2$ then the hook $H_{i,j}$ is contained in $\lambda$ (or more precisely its translation by one row down), and that all the hooks of $\lambda$ show up in this way. Therefore their product yields, by another application of Theorem \ref{hookform}, the quotient $k!/f^{\lambda}$, and we get \[f^{(n-k,\lambda)}=\frac{n!}{\prod_{ij}h_{ij}}=\frac{n! \cdot f^{\lambda}}{k!\cdot\prod_{j=1}^{n-k}h_{1,j}}.\]

Furthermore, for $j>\lambda_{1}$ we have $h_{1,j}=n-k-j+1$, so that we get the product \[\prod_{\lambda_{1}<j\leq n-k}h_{1,j}=(n-k-\lambda_{1})!,\] and canceling it with the corresponding parts of the numerator produces
\begin{equation}
f^{(n-k,\lambda)}=\frac{n\cdot(n-1)\cdot\ldots\cdot(n-k-\lambda_{1}+1) \cdot f^{\lambda}}{k!\cdot\prod_{j=1}^{\lambda_{1}}h_{1,j}}=\frac{f^{\lambda}\cdot\prod_{i=0}^{k+\lambda_{1}-1}(n-i)}{k!\cdot\prod_{j=1}^{\lambda_{1}}h_{1,j}}.
\label{quotient}
\end{equation}
 The numerator now is a constant times the product of $k+\lambda_{1}$ terms that are linear in $n$.

Finally, let $\mu_{j}$ for $1 \leq j\leq\lambda_{1}$ by the number of $\lambda_{i}$'s in $\lambda$ that equal at least $j$. Then $(\mu_{1},\ldots,\mu_{\lambda_{1}})$ is a partition of $k$ (this is the \emph{transpose partition} of $\lambda$), which is of length $\lambda_{1}$ and in which the maximal element $\mu_{1}$ is $\ell(\lambda)$. By looking at the diagram of $(n-k,\lambda)$ we deduce that for every such $j$ we have $h_{1,j}=n-k+\lambda_{1}-j+\mu_{j}$, yielding $\lambda_{1}$ integers satisfying
\[n \geq n-k+\ell(\lambda)=n-k+\mu_{1}>n-k-1+\mu_{2}>\ldots>n-k-\lambda_{1}+1+\mu_{\lambda_{1}}>n-k-\lambda_{1}.\] Their product is thus a product of $\lambda_{1}$ terms that are linear in $n$, all of which show up in the previous numerator. By canceling them out we remain with a constant times the product of $k$ terms that are linear in $n$, yielding indeed a polynomial of degree $k$ in $n$, and note that all the terms $n-i$ with $0 \leq i<k-\ell(\lambda)$ do not show up in our denominator, hence survive after this cancelation as desired. This proves the proposition.
\end{proof}

Now, many results of \cite{[R]} are established by expanding the polynomial from Proposition \ref{fpol} using the basis consisting of $p_{0}=1$ and $\{p_{h}-p_{h-1}\}_{h=1}^{k}$. The advantage of this basis is that the coefficients there are non-negative. We work with the basis $\{p_{h}\}_{h=0}^{k}$, which as we now see, in general produces sums that are based on $\ell(\lambda)+1$ elements rather than $k+1$.
\begin{prop}
Write $f^{(n-k,\lambda)}$ as $\sum_{h=0}^{k}b_{\lambda,h}p_{k}(n)=\sum_{h=0}^{k}b_{\lambda,h}\binom{n}{h}$ with coefficients $b_{\lambda,h}\in\mathbb{Q}$. Then $b_{\lambda,h}\in\mathbb{Z}$ for every $h$, and the coefficients with indices $h \leq k-\ell(\lambda)-1$ all vanish. \label{bdonindex}
\end{prop}
Indeed, Proposition \ref{fpol} and the basis property allows us to write $f^{(n-k,\lambda)}$ as such a sum, and we are interested in the properties of the coefficients.

\begin{proof}
Since $f^{(n-k,\lambda)}$ is integral for large enough $n$ (as the cardinality of a set), the first assertion follows directly from Proposition \ref{coeffsinZ}. Moveover, the substitution of $x=m$ for every integer $0 \leq m \leq k-\ell(\lambda)-1$ annihilates the product from Proposition \ref{fpol}, hence also its multiple $p_{\lambda}$. The second assertion is now a consequence of Lemma \ref{vancoeffs}. This proves the proposition.
\end{proof}

Proposition \ref{bdonindex} and the fact that $p_{h}(n)=\binom{n}{h}$ for non-negative integers $n$ allow us to write \[f^{(n-k,\lambda)}=\sum_{h=k-\ell(\lambda)}^{k}b_{\lambda,h}\binom{n}{h}=\sum_{0}^{\ell(\lambda)}b_{\lambda,k-h}\binom{n}{k-h},\] after a change of index. By setting $a_{\lambda,h}\coloneqq(-1)^{k-h}b_{\lambda,k-h}$, the latter expression becomes
\begin{equation}
f^{(n-k,\lambda)}=\sum_{h=0}^{\ell(\lambda)}(-1)^{h}a_{\lambda,h}\binom{n}{k-h}. \label{expoff}
\end{equation}
We will see in Proposition \ref{poscoeff} below that $a_{\lambda,h}>0$ for all $0 \leq h\leq\ell(\lambda)$ in this convention, hence our choice of sign.

\smallskip

The second property of the numbers $f^{\lambda}$ is the \emph{branching rule}. For stating it, we recall a cell $v=(i,j)$ of a partition $\lambda \vdash k$ is called an \emph{inner corner} of $\lambda$ if $(i+1,j)\notin\lambda$ and $(i,j+1)\notin\lambda$. It is the necessary and sufficient property for the complement of $v$ in $\lambda \setminus v$ to be a partition of $k-1$ (to express which we subtract 1 from $\lambda_{i}$ and leaving the other parts of $\lambda$ invariant), which we denote by $\lambda-v$. We will also denote \[\operatorname{IC}(\lambda)=\{v\in\lambda\;|\;v\text{ is an inner corner of }\lambda\}.\]

Using this notation we state the following result, which is given as, e.g., Lemma 2.8.2 of \cite{[S]}.
\begin{thm}[Branching Rule]
For any partition $\lambda$ we have the equality \[f^{\lambda}=\sum_{v\in\operatorname{IC}(\lambda)}f^{\lambda-v}.\] \label{branching}
\end{thm}
We will indicate the proof of Theorem \ref{branching} in Remark \ref{forBR} below, as we will later use the same reasoning for other results.

We now show that the coefficients from Equation (\ref{expoff}) also satisfy a branching rule, similar to that from Theorem \ref{branching}. Recall that $(1^{k})$ is the column partition of $k$, consisting of $k$ rows, each with a single entry.
\begin{lem}
For $(1^{k})\neq\lambda \vdash k$ and $1 \leq h \leq k-1$, we have the equality \[a_{\lambda,h}=\sum_{v\in\operatorname{IC}(\lambda)}a_{\lambda-v,h}\]. \label{brcoeffs}
\end{lem}

\begin{proof}
If $n$ is large enough, then the inner corners of $(n-k,\lambda)$ are exactly the inner corners of $\lambda$ (pushed one row to the bottom), and the rightmost box on the top row. Observe that Theorem \ref{branching} and the definition of $f^{(n-k,\lambda)}$ then give
\begin{equation}
f^{(n-k,\lambda)}=f^{(n-1-k,\lambda)}+\sum_{v\in\operatorname{IC}(\lambda)}f^{(n-1-(k-1),\lambda-v)}, \label{decomf}
\end{equation}
where we wrote $n-k$ as $(n-1)-(k-1)$ because $\lambda-v \vdash k-1$ for $v\in\operatorname{IC}(\lambda)$.

Now, as $\lambda\neq(1^{k})$, the fact that $\ell(\lambda)<k$ allows us to take the sum from Equation (\ref{expoff}) up to $k-1$, with $a_{\lambda,h}=0$ for all $\ell(\lambda)<h \leq k-1$. A similar argument expresses the right hand side of Equation (\ref{decomf}) as
\[\sum_{h=0}^{k-1}(-1)^{h}a_{\lambda,h}\binom{n-1}{k-h}+\sum_{v\in\operatorname{IC}(\lambda)}\sum_{h=0}^{k-1}(-1)^{h}a_{\lambda-v,h}\binom{n-1}{k-1-h},\]
and rearranging the resulting equality gives
\begin{align*}
\sum_{v\in\operatorname{IC}(\lambda)}\sum_{h=0}^{k-1}(-1)^{h}a_{\lambda-v,h}\binom{n-1}{k-1-h} & =\sum_{h=0}^{k-1}(-1)^{h}a_{\lambda,h}\binom{n}{k-h}-\sum_{h=0}^{k-1}(-1)^{h}a_{\lambda,h}\binom{n-1}{k-h} \\ & =\sum_{h=0}^{k-1}(-1)^{h}a_{\lambda,h}\left[\binom{n}{k-h}-\binom{n-1}{k-h}\right] \\ & =\sum_{h=0}^{k-1}(-1)^{h}a_{\lambda,h}\binom{n-1}{k-1-h},
\end{align*}
by Pascal's identity. Viewing both sides as the substitution $x=n-1$ (for arbitrary large $n$) in linear combinations of the polynomials $\{p_{k-1-h}\}_{h=0}^{k-1}$, the resulting equality follows from the basis property of these polynomials. This proves the lemma.
\end{proof}
Note that the proof of Lemma \ref{brcoeffs} shows that the equality also holds for $\ell(\lambda)<h \leq k-1$ (with the extension by 0), and with it to all $h>\ell(\lambda)$, as both sides indeed vanish. This proof is closely related to that of Theorem B of \cite{[R]}.

The case $\lambda=(1^{k})$, omitted in Lemma \ref{brcoeffs}, will serve as our basic case in what follows. The result for which is as follows.
\begin{lem}
If $\lambda=(1^{k})$ then $a_{\lambda,h}=1$ for all $0 \leq h \leq k$. \label{onecol}
\end{lem}

\begin{proof}
We follows the proof of Proposition \ref{fpol}. The number $f^{\lambda}$ is 1 for $\lambda=(1^{k})$, and then using the fact that $\lambda_{1}=1$ and $h_{1,1}=n$, Equation (\ref{quotient}) reduces to
\[f^{(n-k,\lambda)}=\frac{\prod_{i=0}^{k}(n-i)}{k! \cdot h_{1,1}}=\frac{\prod_{i=1}^{k}(n-i)}{k!}=\binom{n-1}{k}.\] As Pascal's identity expresses, for each $1 \leq t \leq k$, the coefficient $\binom{n-1}{t}$ as $\binom{n}{t}-\binom{n-1}{t-1}$, it follows by a simple induction that $\binom{n-1}{k}=\sum_{h=0}^{k}(-1)^{h}\binom{n}{k-h}$. Comparing this result with Equation (\ref{expoff}) using the basis property yields the asserted result.
\end{proof}

\smallskip

We already know that $a_{\lambda,h}=0$ for every $h>\ell(\lambda)$. For the other coefficients, we obtain the following result.
Using this example we can now prove the following proposition.
\begin{prop}
For every partition $\lambda \vdash k$ and $0 \leq h\leq\ell(\lambda)$ we have $a_{\lambda,h}>0$. \label{poscoeff}
\end{prop}

\begin{proof}
We argue by induction on $k$. When $k=1$ there is only one partition $\lambda=(1)$, for which Lemma \ref{onecol} implies that $a_{\lambda,0}=a_{\lambda,1}=1>0$, as desired.

So assume that the assertion holds for every partition of $k-1$, and take $\lambda \vdash k$. When $\lambda=(1^{k})$ we again apply Lemma \ref{onecol} to get $a_{\lambda,h}=1>0$ for every $0 \leq h \leq k=\ell(\lambda)$. We thus take $(1^{k})\neq\lambda \vdash k$ and an index $0 \leq h \leq \ell(\lambda)$, and apply Lemma \ref{brcoeffs}. Our induction hypothesis implies that $a_{\lambda-v,h}\geq0$ for every $v\in\operatorname{IC}(\lambda)$, with a strict inequality if $h\leq\ell(\lambda-v)$.

It thus remains to show that there is some $v$ for which $h>\ell(\lambda-v)$. For this, note that $\lambda_{1}\geq2$ (since $\lambda\neq(1^{k})$), take $i$ to be the maximal index with $\lambda_{i}=\lambda_{1}$, and set $j=\lambda_{1}$. Then $v=(i,j)$ is in $\operatorname{IC}(\lambda)$, and it is clear that $\ell(\lambda-v)=\ell(\lambda) \geq h$. Hence this summand in the expansion from Lemma \ref{brcoeffs} is strictly positive, and with it so is the entire sum $a_{\lambda,h}$. This proves the proposition.
\end{proof}

\section{The Coefficients and Counting Tableaux \label{ResTab}}

In this section we show that the coefficients from Equation (\ref{expoff}), whose positivity is ensured by Proposition \ref{poscoeff}, count certain standard Young tableaux.

\smallskip

Our first result, which is equivalent to parts $(ii)$ and $(iii)$ of Theorem D of \cite{[R]}, is as follows.
\begin{prop}
For every partition $\lambda$ we have $a_{\lambda,0}=a_{\lambda,1}=f^{\lambda}$. \label{01coeffs}
\end{prop}

\begin{proof}
Also here we prove the results by induction on the number $k$ such that $\lambda \vdash k$. When $k=1$ and $\lambda=(1)$ we already saw in the proof of Proposition \ref{poscoeff} that both these coefficients equal 1, which is also the value of $f^{\lambda}$.

So let $\lambda \vdash k$, and assume that these equalities hold for every partition of $k-1$. If $\lambda=(1^{k})$ then both values are 1 (via Lemma \ref{onecol}), and we know that $f^{\lambda}=1$ as well. Otherwise the induction hypothesis shows that $a_{\lambda-v,0}=a_{\lambda-v,1}=f^{\lambda-v}$ for every $v\in\operatorname{IC}(\lambda)$, and after using Lemma \ref{brcoeffs} and Theorem \ref{branching} to express $a_{\lambda,0}$, $a_{\lambda,1}$, and $f^{\lambda}$ as the respective sums over $v\in\operatorname{IC}(\lambda)$, the result for $\lambda$ follows as well. This proves the proposition.
\end{proof}

To continue exploring the coefficients in general, we will need some notation and a few definitions. If $\lambda \vdash k$, $T\in\operatorname{SYT}(\lambda)$, and $1 \leq m \leq k$, then we write $R_{T}(m)$ and $C_{T}(m)$ for the row and column respectively of the box containing $m$ in $T$, and $v_{T}(m)$ is the box $\big(R_{T}(m),C_{T}(m)\big)$ itself. Namely, if $m$ lies in the box $(i,j)$, then $R_{T}(m)=i$, $C_{T}(m)=j$, and $v_{T}(m)=(i,j)$.

We now make the following definition.
\begin{defn}
Take $\lambda \vdash k$, an integer $0 \leq h \leq k$, and another integer $0\leq\alpha \leq k-h$. We say that a tableau $T\in\operatorname{SYT}(\lambda)$ satisfies the \emph{$(h,\alpha)$-condition} if the entries from $1+\alpha$ to $h+\alpha$ lie in increasing rows, namely we have $R_{T}(i+1+\alpha)>R_{T}(i+\alpha)$ for all $1 \leq i<h$. We then define \[\operatorname{SYT}_{h,\alpha}(\lambda):=\{T\in\operatorname{SYT}(\lambda)\;|\;T\text{ satisfies the }(h,\alpha)-\text{condition}\},\mathrm{\ and\ }f_{h,\alpha}^{\lambda}:=|\operatorname{SYT}_{h,\alpha}(\lambda)|.\] We will also use the shorthand $\operatorname{SYT}_{h}(\lambda):=\operatorname{SYT}_{h,0}(\lambda)$ (hence also $f_{h}^{\lambda}:=f_{h,0}^{\lambda}$), counting the number of tableaux in which the first $h$ rows begin with their respective row numbers. \label{SYThalpha}
\end{defn}
Note that if $h>\ell(\lambda)$ then $\operatorname{SYT}_{h,\alpha}(\lambda)=\emptyset$, since there are not enough rows to contain the numbers for the $(h,\alpha)$-condition from Definition \ref{SYThalpha} to hold. Also recall that for $h=1$, and certainly for $h=0$, the $(h,\alpha)$-condition holds for every tableau, so that $\operatorname{SYT}_{0,\alpha}(\lambda)=\operatorname{SYT}_{1,\alpha}(\lambda)$ for every $\alpha$.

To give an example, let $\lambda=(3,3,2,2)\vdash10$, where we have we have
\begin{align*}
\begin{bmatrix}1 & 2 & 3 \\ 4 & 5 & \boldsymbol{6} \\ \boldsymbol{7} & \boldsymbol{8} \\ \boldsymbol{9} & 10\end{bmatrix} & \notin\operatorname{SYT}_{4,5}(\lambda), & \begin{bmatrix}1 & 2 & \boldsymbol{6} \\ 3 & 5 & \boldsymbol{7} \\ 4 & \boldsymbol{8} \\ \boldsymbol{9} & 10\end{bmatrix}\in\operatorname{SYT}_{4,5}(\lambda), \\ \begin{bmatrix}\boldsymbol{1} & \boldsymbol{2} & \boldsymbol{3} \\ \boldsymbol{4} & 5 & 6 \\ 7 & 8 \\ 9 & 10\end{bmatrix} & \notin\operatorname{SYT}_{4}(\lambda), & \begin{bmatrix}\boldsymbol{1} & 5 & 6 \\ \boldsymbol{2} & 7 & 8 \\ \boldsymbol{3} & 9 \\ \boldsymbol{4} & 10\end{bmatrix}\in\operatorname{SYT}_{4}(\lambda).
\end{align*}

Before we relate the objects from Definition \ref{SYThalpha} with our coefficients of interest, we consider a result involving the first non-trivial case $h=2$. Recall the transpose partition that was mentioned in the proof of Proposition \ref{fpol}, which we shall denote by $\lambda^{t}$. Similarly, if $\lambda$ is filled to a tableau $T$, transposing the boxes with the numbers filling them produces a tableau of shape $\lambda^{t}$, which we denote by $T^{t}$. It is clearly standard if and only if $T$ is such.
\begin{prop}
For every partition $\lambda \vdash k$ and $0\leq\alpha \leq k-2$ we have the equality \[f_{2,\alpha}^{\lambda}+f_{2,\alpha}^{\lambda^{t}}=f^{\lambda}.\] \label{SYT2}
\end{prop}

\begin{proof}
Consider a tableau $T\in\operatorname{SYT}(\lambda)$. If $R_{T}(2+\alpha)>R_{T}(1+\alpha)$, then we must have the inequality $C_{T}(2+\alpha) \leq C_{T}(1+\alpha)$, since if $C_{T}(2+\alpha)>C_{T}(1+\alpha)$ as well then the box $\big(R_{T}(2+\alpha),C_{T}(1+\alpha)\big)$, or $\big(R_{T}(1+\alpha),C_{T}(2+\alpha)\big)$ would have to contain a number that is larger than $1+\alpha$ and smaller than $2+\alpha$, and there is no such integer.

Assuming now that $R_{T}(2+\alpha) \leq R_{T}(1+\alpha)$, so that the box $\big(R_{T}(2+\alpha),C_{T}(1+\alpha)\big)$ either equals $v_{T}(1+\alpha)$ or lies above it, and therefore it is filled with a number that is bounded by $1+\alpha$ (since $T$ is standard). It follows that $v_{T}(2+\alpha)$ must lie to the right of that box, so that $C_{T}(2+\alpha)>C_{T}(1+\alpha)$.

Now, transposition is a bijection between $\operatorname{SYT}(\lambda)$ and $\operatorname{SYT}(\lambda^{t})$. Moreover, it takes a table $T$ with $R_{T}(2+\alpha)>R_{T}(1+\alpha)$ and $C_{T}(2+\alpha) \leq C_{T}(1+\alpha)$ (namely an element of $\operatorname{SYT}_{2,\alpha}(\lambda)$) to $T^{t}$ with $C_{T^{t}}(2+\alpha)>C_{T^{t}}(1+\alpha)$ and $R_{T^{t}}(2+\alpha) \leq R_{T^{t}}(1+\alpha)$, while if $R_{T}(2+\alpha) \leq R_{T}(1+\alpha)$ and $C_{T}(2+\alpha)>C_{T}(1+\alpha)$ then $C_{T^{t}}(2+\alpha) \leq C_{T^{t}}(1+\alpha)$ and $R_{T^{t}}(2+\alpha)>R_{T^{t}}(1+\alpha)$, so that $T^{t}\in\operatorname{SYT}_{2,\alpha}(\lambda^{t})$.

In total, the set $\operatorname{SYT}(\lambda)$, of cardinality $f^{\lambda}$, is the disjoint union of $\operatorname{SYT}_{2,\alpha}(\lambda)$, of cardinality $f_{2,\alpha}^{\lambda}$, and a complement which is in bijection with the set $\operatorname{SYT}_{2,\alpha}(\lambda^{t})$, of cardinality $f_{2,\alpha}^{\lambda^{t}}$, and the required equality follows. This proves the proposition.
\end{proof}

\begin{cor}
If $k\geq2$ and $\lambda$ is invariant under transposition, then $f^{\lambda}$ is even as the union of pairs of transposed tableaux, and every choice of $\alpha$ picks up one representative from each such pair. \label{symev}
\end{cor}

\begin{proof}
After we fix $0\leq\alpha \leq k-2$, Proposition \ref{SYT2} expresses $f^{\lambda}$ as the sum of $f_{2,\alpha}^{\lambda}$ and $f_{2,\alpha}^{\lambda^{t}}$. When $\lambda=\lambda^{t}$, these numbers coincide, so that $f^{\lambda}=2f_{2,\alpha}^{\lambda}$ is indeed even, and the proof of Proposition \ref{SYT2} decomposes $\operatorname{SYT}(\lambda)$ as the disjoint union of $\operatorname{SYT}_{2,\alpha}(\lambda)$ and its image under transposition. The asserted choice picks, from every pair, the element of the pair that lies in $\operatorname{SYT}_{2,\alpha}(\lambda)$. This proves the corollary.
\end{proof}
The fact that $f^{\lambda}$ is even in Corollary \ref{symev} follows immediately from the fact that transposition is an operation of order 2 which acts on it faithfully, but the corollary also produces, for every $\alpha$, a choice of a representative of each orbit. Note that the condition $k\geq2$ is necessary for such $\alpha$ to exist, and indeed the unique partition $\lambda=(1)$ of $k=1$ is symmetric the corresponding $f^{\lambda}=1$ is odd. We will later see that the numbers showing up in Proposition \ref{SYT2} and Corollary \ref{symev} are independent of the choice of $0\leq\alpha \leq k-2$.

\smallskip

We now turn back to properties of $f_{h,\alpha}^{\lambda}$ for general $h$ and $\alpha$. Some proofs will be emulating that of Theorem \ref{branching} in \cite{[S]}, which we now explain for fixing some ideas and notation.
\begin{rem}
Using our notation, we can indicate how Theorem \ref{branching} is proved. When $\lambda \vdash k$ and $T\in\operatorname{SYT}(\lambda)$, we must have $v_{T}(k)\in\operatorname{IC}(\lambda)$, so that if we set \[\operatorname{SYT}^{v}(\lambda):=\{T\in\operatorname{SYT}(\lambda)\;|\;v_{T}(k)=v\}\] for every $v\in\operatorname{IC}(\lambda)$ then we immediately get $\operatorname{SYT}(\lambda)=\bigsqcup_{v\in\operatorname{IC}(\lambda)}\operatorname{SYT}^{v}(\lambda)$ (the union being disjoint since $v_{T}(k)$ determines the index). The fact that removing $v$ and $k$ yields a bijection between $\operatorname{SYT}^{v}(\lambda)$ and $\operatorname{SYT}(\lambda-v)$, with the inverse being adding $v$ (which is an \emph{outer corner} of $\lambda-v$, making this operation legitimate for staying in the realm of Young diagrams) and putting $k$ into it, immediately yields Theorem \ref{branching}. \label{forBR}
\end{rem}

We will now apply the ideas from Remark \ref{forBR} to obtain the branching rule for the numbers counting the objects from Definition \ref{SYThalpha}.
\begin{lem}
For $\lambda \vdash k$, $0 \leq h\leq\ell(\lambda)$, and $0\leq\alpha<k-h$, we have the equality \[f_{h,\alpha}^{\lambda}=\sum_{v\in\operatorname{IC}(\lambda)}f_{h,\alpha}^{\lambda-v}.\] \label{halphabr}
\end{lem}
In fact, Lemma \ref{halphabr} holds for every $h$, since when $\ell(\lambda)<h<k$ both sides vanish (and for $h \geq k$ the set of $\alpha$'s is empty).

\begin{proof}
As in Remark \ref{forBR}, consider $v\in\operatorname{IC}(\lambda)$, and set \[\operatorname{SYT}_{h,\alpha}^{v}(\lambda):=\{T\in\operatorname{SYT}_{h,\alpha}(\lambda)\;|\;v_{T}(k)=v\},\] to again obtain the description $\operatorname{SYT}_{h,\alpha}(\lambda)=\bigsqcup_{v\in\operatorname{IC}(\lambda)}\operatorname{SYT}_{h,\alpha}^{v}(\lambda)$ as a disjoint union. Now, when $1 \leq h<k$ and $0\leq\alpha<k-h$, the $(h,\alpha)$-condition from Definition \ref{SYThalpha} does not involve the number $k$. Therefore restricting the bijection between $\operatorname{SYT}^{v}(\lambda)$ and $\operatorname{SYT}(\lambda-v)$, as described in Remark \ref{forBR}, to $\operatorname{SYT}_{h,\alpha}^{v}(\lambda)$ preserves the condition in question on both sides, and reduces to a bijection between the latter set and $\operatorname{SYT}_{h,\alpha}(\lambda-v)$. Gathering these together, we obtain \[|\operatorname{SYT}_{h,\alpha}(\lambda)|=\left|\bigsqcup_{v\in\operatorname{IC}(\lambda)}\operatorname{SYT}_{h,\alpha}^{v}(\lambda)\right|= \sum_{v\in\operatorname{IC}(\lambda)}|\operatorname{SYT}_{h,\alpha}^{v}(\lambda)|=\sum_{v\in\operatorname{IC}(\lambda)}|\operatorname{SYT}_{h,\alpha}(\lambda-v)|,\]
which is precisely the desired equality. This proves the lemma.
\end{proof}

We can now state and prove our first main result.
\begin{thm}
For every $\lambda \vdash k$ and any $0 \leq h\leq\ell(\lambda)$, the coefficient $a_{\lambda,h}$ from Equation (\ref{expoff}) equals the cardinality $f_{h}^{\lambda}$ of $\operatorname{SYT}_{h}(\lambda)$ from Definition \ref{SYThalpha}. \label{coeffSYT}
\end{thm}
Also in Theorem \ref{coeffSYT}, the assertion extends to $h>\ell(\lambda)$ as well, since we saw in Equation (\ref{expoff}) that $a_{\lambda,h}=0$ for such $h$, and as already mentioned above, it follows directly from Definition \ref{SYThalpha} that $\operatorname{SYT}_{h}(\lambda)$ is empty in that case, hence $f_{h}^{\lambda}=0$ as well.

\begin{proof}
We first consider the case where $\lambda=(1^{k})$. Then $\operatorname{SYT}(\lambda)$ consists of a single tableau, which clearly satisfies that $(h,0)$-condition (and, in fact, the $(h,\alpha)$-condition for every $0\leq\alpha \leq k-h$) from Definition \ref{SYThalpha}, so that $f_{h}^{\lambda}=1$ for every $0 \leq h \leq k$. As Lemma \ref{onecol} implies that $a_{\lambda,h}=1$ for every such $h$, the asserted equality follows in this case.

This proves, in particular, the result when $k=1$ as above, so we have the base case for arguing by induction over $k$. Assume that the result holds for every partition of $k-1$ and every $h$, and take some partition $\lambda \vdash k$, which we may assume to be different from $(1^{k})$ as this case was already established. We can also assume that $0 \leq h\leq\ell(\lambda)<k$, since we noted already that otherwise the equality holds in a trivial manner.

Now, the fact that $\lambda\neq(1^{k})$ allows us to invoke Lemma \ref{brcoeffs}, and we can apply the induction hypothesis to all its summands (indeed, while the bound $h\leq\ell(\lambda)$ might not be the one applied for $\lambda-v$ in some cases, the induction hypothesis still applies as an equality of zeros in the remaining case). It follows that $h<k$, and then $\alpha=0<k-h$, meaning that the conditions of Lemma \ref{halphabr} are also satisfied. Altogether we obtain the equalities \[a_{\lambda,h}=\sum_{v\in\operatorname{IC}(\lambda)}a_{\lambda-v,h}=\sum_{v\in\operatorname{IC}(\lambda)}f_{h}^{\lambda-v}=f_{h}^{\lambda},\] as desired. This proves the theorem.
\end{proof}
In fact, by changing the roles of the numbers we can take an arbitrary $\lambda=(\lambda_{1},\ldots,\lambda_{\ell(\lambda)}) \vdash k$, and set $\mu:=(\lambda_{2},\ldots,\lambda_{\ell(\lambda)})$, so that $\lambda$ can be written as $(\lambda_{1},\mu)$. Using Equation (\ref{expoff}) and Theorem \ref{coeffSYT}, we can express $f^{\lambda}$ as \[f^{\lambda}=f^{(\lambda_{1},\mu)}=\sum_{h=0}^{\ell(\lambda)-1}(-1)^{h}a_{\mu,h}\binom{k}{k-h}=\sum_{h=0}^{\ell(\lambda)-1}(-1)^{h}f_{h}^{\mu}\binom{k}{k-h}.\]

\section{Bijections Between Different Values of $\alpha$ \label{Bijections}}

In the first draft of this paper, we could compare, in a more complicated manner, the coefficient $a_{\lambda,h}$ from Equation (\ref{expoff}) with the number $f_{h,k-h}^{\lambda}$. It was then brought to our attention that $f_{h,\alpha}^{\lambda}$ is independent of $\alpha$ (as long as $\lambda \vdash k \geq h$ and $0\leq\alpha \leq k-h$), and thus the comparison of $a_{\lambda,h}$ with $f_{h}^{\lambda}=f_{h,0}^{\lambda}$ was much easier. In this section we prove this independence of $\alpha$ by constructing explicit bijections between $\operatorname{SYT}_{h,\alpha}(\lambda)$ and $\operatorname{SYT}_{h,\alpha-1}(\lambda)$ for every $0\leq\alpha<h-k$.

\smallskip

The naive approach for attempting to prove this fact would be to apply induction via Lemma \ref{halphabr} for obtaining an equality of the sort
\[f_{h,\alpha}^{\lambda}=\sum_{v\in\operatorname{IC}(\lambda)}f_{h,\alpha}^{\lambda-v}=\sum_{v\in\operatorname{IC}(\lambda)}f_{h,\alpha-1}^{\lambda-v}=f_{h,\alpha-1}^{\lambda}.\] However, Lemma \ref{halphabr} makes the assumption that $\alpha<k-h$, so that such an argument does not work when $\alpha=k-h$, or equivalently $k=h+\alpha$. Our proof will indeed be inductive of this sort, but after we make the necessary constructions for the base case $k=h+\alpha$. Once we show that if $f_{h,\alpha}^{\lambda}=f_{h,\alpha-1}^{\lambda}$ for $\lambda \vdash k=h+\alpha$, the inductive argument above will establish the result for all $\lambda \vdash k>h+\alpha$.

We begin with the following construction.
\begin{defn}
Assume that $h>0$ and $\alpha>0$ are integers, and take a partition $\lambda \vdash h+\alpha$. We define the map $\varphi_{\lambda,h,\alpha}$ from $\operatorname{SYT}_{h,\alpha}(\lambda)$ to tableaux of shape $\lambda$ as follows. If $T\in\operatorname{SYT}_{h,\alpha}(\lambda)$ also lies in $\operatorname{SYT}_{h,\alpha-1}(\lambda)$, namely $R_{T}(\alpha)<R_{T}(1+\alpha)$, then set $\varphi_{\lambda,h,\alpha}(T):=T$. Otherwise the set of $1 \leq i \leq h$ such that $R_{T}(\alpha) \geq R_{T}(i+\alpha)$ is non-empty, and we can define
\[q:=\max\{1 \leq i \leq h\;|\;R_{T}(i+\alpha) \leq R_{T}(\alpha)\}.\] With this $q$ we define $\varphi_{\lambda,h,\alpha}(T)$ as the tableau of shape $\lambda$ in which for $1 \leq m \leq h+\alpha$ we have
\begin{align*}
v_{\varphi_{\lambda,h,\alpha}(T)}(m) & =\begin{cases} v_{T}(m) & 1 \leq m\leq\alpha-1, \\ v_{T}(m+1) & \alpha \leq m \leq q-2+\alpha\mathrm{\ \ or\ \ }q+\alpha \leq m \leq h-1+\alpha, \\ v_{T}(\alpha) & m=q-1+\alpha, \\ v_{T}(q+\alpha) & m=h+\alpha. \end{cases}
\end{align*} \label{mapredalpha}
\end{defn}
Note that each of the boxes $v_{T}(m)$ with $1 \leq m \leq h+\alpha$ shows up once in Definition \ref{mapredalpha}, so that $\varphi_{\lambda,h,\alpha}(T)$ is indeed a tableau of shape $\lambda$, but at the moment we do not know yet that it is in $\operatorname{SYT}_{h,\alpha-1}(\lambda)$, or even standard. While this is clear if $T$ also lies in $\operatorname{SYT}_{h,\alpha-1}(\lambda)$, in the other cases we will use the following observation.
\begin{lem}
Given $h>0$, $\alpha>0$, $\lambda \vdash h+\alpha$, and $T\in\operatorname{SYT}_{h,\alpha}(\lambda)\setminus\operatorname{SYT}_{h,\alpha-1}(\lambda)$, let $q$ be as in Definition \ref{mapredalpha}. Then $v_{T}(q+\alpha)$ is an inner corner of $\lambda$. \label{qalphaIC}
\end{lem}

\begin{proof}
First we note that as $T\in\operatorname{SYT}_{h,\alpha}(\lambda)$, all the numbers that are larger than $q+\alpha$ show up in lower rows, meaning that this number is the largest in its row and thus fills its rightmost box (this is independent of the value of $q$, and thus holds for $q+\alpha$ for every $1 \leq i \leq h$). It is thus an inner corner of $\lambda$ if and only if there is no box just below it, namely if and only if $\big(R_{T}(q+\alpha)+1,C_{T}(q+\alpha)\big)\not\in\lambda$.

So assume that $\big(R_{T}(q+\alpha)+1,C_{T}(q+\alpha)\big)\in\lambda$, and it thus must be filled, in the standard tableau $T$, by a number that is larger than $q+\alpha$. Since any number that is larger than $q+1+\alpha$ must be in a row below $R_{T}(q+1+\alpha)>R_{T}(q+\alpha)$, it cannot be in the row $R_{T}(q+\alpha)+1$, and thus the only option for this box is to contain the number $q+1+\alpha$.

But recall that $R_{T}(\alpha) \geq R_{T}(q+\alpha)$, and $q$ is maximal with that property. If this inequality is strict then we obtain $R_{T}(\alpha) \geq R_{T}(q+\alpha)+1=R_{T}(q+1+\alpha)$, which contradicts the maximality of $q$, so we must have If $R_{T}(\alpha)=R_{T}(q+\alpha)$. Then we have $C_{T}(\alpha)<C_{T}(q+\alpha)$ (as $T$ is standard and they are in the same row), and consider the box $\big(R_{T}(q+\alpha)+1,C_{T}(\alpha)\big)\in\lambda$. It lies below $v_{T}(\alpha)$, thus must be filled with a larger entry. But it also lies to the left of $v_{T}(q+1+\alpha)$, so must be filled with a smaller entry. But $q+\alpha$ lies in an upper row, and the $(h,\alpha)$-condition on $T$ implies that all the numbers between $1+\alpha$ and $q+\alpha$ lie in rows that are above $R_{T}(q+\alpha)+1$. Thus we cannot fill such $T$ to be a standard tableau, and from this contradiction we obtain that $\big(R_{T}(q+\alpha)+1,C_{T}(q+\alpha)\big)\not\in\lambda$ as desired. This proves the lemma.
\end{proof}

We can now prove that $\varphi_{\lambda,h,\alpha}$ has the desired property.
\begin{prop}
With $h$, $\alpha$, and $\lambda$ as above, $\varphi_{\lambda,h,\alpha}$ takes every $T\in\operatorname{SYT}_{h,\alpha}(\lambda)$ to an element of $\operatorname{SYT}_{h,\alpha-1}(\lambda)$. \label{decalpha}
\end{prop}

\begin{proof}
The result is obvious when $T\in\operatorname{SYT}_{h,\alpha-1}(\lambda)$ by definition, so assume that $T$ is in $\operatorname{SYT}_{h,\alpha}(\lambda)$ but not in $\operatorname{SYT}_{h,\alpha-1}(\lambda)$, and let $q$ be as in Definition \ref{mapredalpha}. We will first prove the inequalities from the $(h,\alpha-1)$-condition from Definition \ref{SYThalpha} for $\varphi_{\lambda,h,\alpha}(T)$, and then show that the latter tableau is also standard.

For every $1 \leq j \leq q-2$ and for every $q+1 \leq j \leq h$ (one of these ranges may be empty in case $q$ is close to 1 or to $h$), Definition \ref{mapredalpha} and the fact that $T\in\operatorname{SYT}_{h,\alpha}(\lambda)$ give the inequality \[R_{\varphi_{\lambda,h,\alpha}(T)}(j+\alpha-1)=R_{T}(j+\alpha)<R_{T}(j+1+\alpha)=R_{\varphi_{\lambda,h,\alpha}(T)}(j+\alpha),\] as desired. For the remaining values of $j$, we recall that the maximality of $q$ implies that
\begin{equation}
R_{T}(q-1+\alpha)<R_{T}(q+\alpha) \leq R_{T}(\alpha)<R_{T}(q+1+\alpha), \label{ineqvarphi}
\end{equation}
with the leftmost inequality holding only if $q\geq2$. As the leftmost number is $R_{\varphi_{\lambda,h,\alpha}(T)}(q+\alpha-2)$ (when $q\geq2$), the rightmost one is $R_{\varphi_{\lambda,h,\alpha}(T)}(q+\alpha)$, and one of the intermediate terms is $R_{\varphi_{\lambda,h,\alpha}(T)}(q-1+\alpha)$ (all by Definition \ref{mapredalpha}), the inequalities for the remaining value $q$ of $j$, as well as $q-1$ in case $q\geq2$, follow.

Next, Definition \ref{mapredalpha} implies that all the boxes in $T$ that were filled by numbers that are smaller than $\alpha$ still remain with their numbers also in $\varphi_{\lambda,h,\alpha}(T)$, so satisfy the standard condition. Moreover, the proof of Lemma \ref{qalphaIC} shows that the boxes that were filled by the numbers between $1+\alpha$ and $h+\alpha$ are the rightmost of their rows, and for $\alpha$ either the same happens (by the same argument), or we have $R_{T}(\alpha)=R_{T}(q+\alpha)$ and $\alpha$ lies directly to the left of $q+\alpha$. It follows that also in $\varphi_{\lambda,h,\alpha}(T)$ the inequalities comparing the places of boxes containing numbers that are at least $\alpha$ with those having numbers that are smaller than $\alpha$, that are required for $\varphi_{\lambda,h,\alpha}(T)$ to be standard, are also satisfied.

Finally, we recall from Definition \ref{mapredalpha} and Lemma \ref{qalphaIC} that $v_{\varphi_{\lambda,h,\alpha}(T)}(h+\alpha)=v_{T}(q+\alpha)$ is in $\operatorname{IC}(\lambda)$, and hence $\varphi_{\lambda,h,\alpha}(T)$ is standard if and only if removing that box and the number $h+\alpha$ gives a standard tableau (see Remark \ref{forBR}). But combining the inequalities from the $(h,\alpha-1)$-condition from Definition \ref{SYThalpha} with the previous paragraph shows that the latter tableau is indeed standard, hence so is $\varphi_{\lambda,h,\alpha}(T)$. This completes the proof of the proposition.
\end{proof}

We now construct a map in the opposite direction.
\begin{defn}
Take again $h>0$, $\alpha>0$, and a partition $\lambda \vdash h+\alpha$, and for $T\in\operatorname{SYT}_{h,\alpha-1}(\lambda)$ we define the tableau $\phi_{\lambda,h,\alpha}(T)$, also of shape $\lambda$, as follows. When $T$ is also in $\operatorname{SYT}_{h,\alpha}(\lambda)$, we define $\phi_{\lambda,h,\alpha}(T):=T$, and otherwise there is $0 \leq i<h$ such that $R_{T}(h+\alpha) \leq R_{T}(i+\alpha)$, and we set \[p:=\min\{0 \leq i \leq h-1\;|\;R_{T}(h+\alpha) \leq R_{T}(i+\alpha)\}.\] Using this value of $p$, take any $1 \leq m \leq h+\alpha$, and define
\begin{align*}
v_{\phi_{\lambda,h,\alpha}(T)}(m) & =\begin{cases} v_{T}(m) & 1 \leq m \leq \alpha-1, \\ v_{T}(m-1) & 1+\alpha \leq m \leq p+\alpha\mathrm{\ \ or\ \ }p+2+\alpha \leq m \leq h+\alpha, \\ v_{T}(h+\alpha) & m=p+1+\alpha, \\ v_{T}(p+\alpha) & m=\alpha. \end{cases}
\end{align*} \label{mapincalpha}
\end{defn}
Also in Definition \ref{mapincalpha}, all the boxes $v_{T}(m)$, $1 \leq m \leq h+\alpha$ appear once in the definition, thus yielding a tableau of shape $\lambda$. We now prove the analogue of Proposition \ref{decalpha} for $\phi_{\lambda,h,\alpha}$.
\begin{prop}
Given $\alpha$, $h$, and $\lambda$ as always, we have $\phi_{\lambda,h,\alpha}(T)\in\operatorname{SYT}_{h,\alpha}(\lambda)$ for every tableau $T\in\operatorname{SYT}_{h,\alpha-1}(\lambda)$. \label{enlalpha}
\end{prop}

\begin{proof}
This is clearly the case when $T\in\operatorname{SYT}_{h,\alpha}(\lambda)$, so take $T\in\operatorname{SYT}_{h,\alpha-1}(\lambda)\setminus\operatorname{SYT}_{h,\alpha}(\lambda)$, with $p$ as given in Definition \ref{mapincalpha}. Also here we begin with establishing the $(h,\alpha)$-condition for $\phi_{\lambda,h,\alpha}(T)$, and then show that it is also standard.

Given $1 \leq j \leq p-1$ or $p+2 \leq j \leq h$, Definition \ref{mapincalpha} and the $(h,\alpha-1)$-condition on $T$ yield \[R_{\phi_{\lambda,h,\alpha}(T)}(j+\alpha)=R_{T}(j-1+\alpha)<R_{T}(j+\alpha)=R_{\phi_{\lambda,h,\alpha}(T)}(j+1+\alpha)\] (also here, for $p$ close to 1 or to $h$ one of the ranges can be empty). We also get from the minimality of $p$ that
\begin{equation}
R_{T}(p-1+\alpha)<R_{T}(h+\alpha) \leq R_{T}(p+\alpha)<R_{T}(p+1+\alpha) \label{ineqphi}
\end{equation}
(the one on the right holding only if $p \leq h-2$), where Definition \ref{mapincalpha} shows that the smallest one here is $R_{\phi_{\lambda,h,\alpha}(T)}(p+\alpha)$, the largest one is $R_{\phi_{\lambda,h,\alpha}(T)}(p+2+\alpha)$ when $p \leq h-2$, and $R_{\phi_{\lambda,h,\alpha}(T)}(p+1+\alpha)$ lies in between. The desired condition is thus established.

Next, the boxes including the numbers between 1 and $\alpha-1$ are unaffected by $\phi_{\lambda,h,\alpha}$, and the larger ones are mixed, so that the inequalities involving the smaller numbers hold for $\phi_{\lambda,h,\alpha}(T)$ because $T$ is standard. Moreover, all the boxes with the larger numbers are the rightmost in their rows, except for $v_{T}(p+\alpha)$ in case its row coincides with $R_{T}(h+\alpha)$. The inequalities from the $(h,\alpha)$-condition leaves only the relation between the box $v_{\phi_{\lambda,h,\alpha}(T)}(\alpha)$ and the one with larger numbers to be verified.

Now, Definition \ref{mapincalpha} transforms Equation (\ref{ineqphi}) into
\begin{equation}
R_{\phi_{\lambda,h,\alpha}(T)}(p+\alpha)<R_{\phi_{\lambda,h,\alpha}(T)}(p+1+\alpha) \leq R_{\phi_{\lambda,h,\alpha}(T)}(\alpha)<R_{\phi_{\lambda,h,\alpha}(T)}(p+2+\alpha), \label{ineqafterphi}
\end{equation}
so that we only have to compare $v_{\phi_{\lambda,h,\alpha}(T)}(\alpha)$ with $v_{\phi_{\lambda,h,\alpha}(T)}(p+1+\alpha)$, as well as $v_{\phi_{\lambda,h,\alpha}(T)}(p+\alpha)$ in case the inequality in the middle is an equality. But we recall, e.g., from the proof of Proposition \ref{SYT2}, that in equality $R_{T}(p+\alpha) \geq R_{T}(h+\alpha)$ from Equation (\ref{ineqphi}) implies $C_{T}(p+\alpha)<C_{T}(h+\alpha)$, which becomes $C_{\phi_{\lambda,h,\alpha}(T)}(\alpha)<C_{\phi_{\lambda,h,\alpha}(T)}(p+1+\alpha)$ by Definition \ref{mapincalpha}. Moreover, in the case where $R_{T}(h+\alpha)=R_{T}(p+\alpha)$ then the fact that $p-1+\alpha$ sits in the rightmost box of its row in $T$ implies that \[C_{\phi_{\lambda,h,\alpha}(T)}(p+\alpha)=C_{T}(p-1+\alpha) \geq C_{T}(h+\alpha)>C_{T}(p+\alpha)=C_{\phi_{\lambda,h,\alpha}(T)}(\alpha)\] as well, and $\phi_{\lambda,h,\alpha}(T)$ is indeed standard. This completes the proof of the proposition.
\end{proof}

\smallskip

Now that Propositions \ref{decalpha} and \ref{enlalpha} show that the maps from Definitions \ref{mapredalpha} and \ref{mapincalpha} go between $\operatorname{SYT}_{h,\alpha}(\lambda)$ and $\operatorname{SYT}_{h,\alpha-1}(\lambda)$, for establishing the equality of the cardinalities of these sets, we need to show that these maps are inverses.
\begin{lem}
Let $h$, $\alpha$, and $\lambda$ be as above. Then for $T\in\operatorname{SYT}_{h,\alpha}(\lambda)\setminus\operatorname{SYT}_{h,\alpha-1}(\lambda)$, with the parameter $q$ from Definition \ref{mapredalpha}, we get $\varphi_{\lambda,h,\alpha}(T)\in\operatorname{SYT}_{h,\alpha-1}(\lambda)\setminus\operatorname{SYT}_{h,\alpha}(\lambda)$, and we have the equality \[\min\{0 \leq i \leq h-1\;|\;R_{\varphi_{\lambda,h,\alpha}(T)}(h+\alpha) \leq R_{\varphi_{\lambda,h,\alpha}(T)}(i+\alpha)\}=q-1.\] Similarly, given $T\in\operatorname{SYT}_{h,\alpha-1}(\lambda)\setminus\operatorname{SYT}_{h,\alpha}(\lambda)$, to which Definition \ref{mapincalpha} associates the parameter $p$, we get $\phi_{\lambda,h,\alpha}(T)\in\operatorname{SYT}_{h,\alpha}(\lambda)\setminus\operatorname{SYT}_{h,\alpha-1}(\lambda)$ and \[\max\{1 \leq i \leq h\;|\;R_{\phi_{\lambda,h,\alpha}(T)}(i+\alpha) \leq R_{\phi_{\lambda,h,\alpha}(T)}(\alpha)\}=p+1.\] \label{relspq}
\end{lem}

\begin{proof}
For the first assertion, Definition \ref{mapredalpha} transforms Equation (\ref{ineqvarphi}) into \[R_{\varphi_{\lambda,h,\alpha}(T)}(q-2+\alpha)<R_{\varphi_{\lambda,h,\alpha}(T)}(h+\alpha) \leq R_{\varphi_{\lambda,h,\alpha}(T)}(q-1+\alpha)\] (the strict one showing up only for $q\geq2$, and we removed the largest value), which shows that the tableau $\varphi_{\lambda,h,\alpha}(T)$, which is in $\operatorname{SYT}_{h,\alpha-1}(\lambda)$ by Proposition \ref{decalpha}, does not lie in $\operatorname{SYT}_{h,\alpha}(\lambda)$. Moreover, it implies that the minimum from Definition \ref{mapincalpha} is at most $q-1$, but cannot be $q-2$ or smaller in case $q\geq2$, so that it indeed equals $q-1$ as desired.

In the setting of the second assertion, Equation (\ref{ineqafterphi}) implies that $\varphi_{\lambda,h,\alpha}(T)$, which is known to be in $\operatorname{SYT}_{h,\alpha}(\lambda)$ via Proposition \ref{enlalpha}, is not in $\operatorname{SYT}_{h,\alpha-1}(\lambda)$, with the maximum from Definition \ref{mapredalpha} being at least $p+1$, but smaller than $p+2$ in case $p \leq h$. This maximum therefore equals $p+1$ as required. This proves the lemma.
\end{proof}

\begin{prop}
For such $h$, $\alpha$, and $\lambda$ the maps $\varphi_{\lambda,h,\alpha}$ and $\phi_{\lambda,h,\alpha}$ are inverses. \label{invmaps}
\end{prop}

\begin{proof}
We need to show that $\phi_{\lambda,h,\alpha}\big(\varphi_{\lambda,h,\alpha}(T)\big)=T$ for every $T\in\operatorname{SYT}_{h,\alpha}(\lambda)$, as well as the equality $\varphi_{\lambda,h,\alpha}\big(\phi_{\lambda,h,\alpha}(T)\big)=T$ for any $T\in\operatorname{SYT}_{h,\alpha-1}(\lambda)$. Since if $T\in\operatorname{SYT}_{h,\alpha}(\lambda)\cap\operatorname{SYT}_{h,\alpha-1}(\lambda)$ we have $\varphi_{\lambda,h,\alpha}(T)=\phi_{\lambda,h,\alpha}(T)=T$ in Definitions \ref{mapredalpha} and \ref{mapincalpha}, both are clear in this case.

We thus consider $T\in\operatorname{SYT}_{h,\alpha}(\lambda)\setminus\operatorname{SYT}_{h,\alpha-1}(\lambda)$, and take $q$ as in Definition \ref{mapredalpha}. Then Lemma \ref{relspq} implies that $\varphi_{\lambda,h,\alpha}(T)\in\operatorname{SYT}_{h,\alpha-1}(\lambda)\setminus\operatorname{SYT}_{h,\alpha}(\lambda)$, and for evaluating its $\phi_{\lambda,h,\alpha}$-image we take $p=q-1$ in Definition \ref{mapincalpha}. This shows that the cell in which a number $1 \leq m \leq h+\alpha$ lies in $\phi_{\lambda,h,\alpha}\big(\varphi_{\lambda,h,\alpha}(T)\big)$ is given by \[\begin{cases} v_{\varphi_{\lambda,h,\alpha}(T)}(m) & 1 \leq m \leq \alpha-1, \\ v_{\varphi_{\lambda,h,\alpha}(T)}(m-1) & 1+\alpha \leq m \leq q-1+\alpha\mathrm{\ \ or\ \ }q+1+\alpha \leq m \leq h+\alpha, \\ v_{\varphi_{\lambda,h,\alpha}(T)}(h+\alpha) & m=q+\alpha, \\ v_{\varphi_{\lambda,h,\alpha}(T)}(q-1+\alpha) & m=\alpha. \end{cases}\] Noting that if $1+\alpha \leq m \leq q-1+\alpha$ or $q+1+\alpha \leq m \leq h+\alpha$ then $\alpha \leq m-1 \leq q-2+\alpha$ or $q+\alpha \leq m \leq h-1+\alpha$, Definition \ref{mapredalpha} transforms this expression into
\[\begin{cases} v_{T}(m) & 1 \leq m\leq\alpha-1, \\ v_{T}(m-1+1) & 1+\alpha \leq m \leq q-1+\alpha\mathrm{\ \ or\ \ }q+1+\alpha \leq m \leq h+\alpha,  \\ v_{T}(q+\alpha) & m=q+\alpha, \\ v_{T}(\alpha) & m=\alpha, \end{cases},\] which equals $v_{T}(m)$ for all $m$ and thus proving the desired equality.

Finally, take $T\in\operatorname{SYT}_{h,\alpha-1}(\lambda)\setminus\operatorname{SYT}_{h,\alpha}(\lambda)$, let $p$ be the number from Definition \ref{mapincalpha}, and then $\phi_{\lambda,h,\alpha}(T)\in\operatorname{SYT}_{h,\alpha}(\lambda)\setminus\operatorname{SYT}_{h,\alpha-1}(\lambda)$ by Lemma \ref{relspq}, and the associated number from Definition \ref{mapredalpha} is $p+1$. It follows that given $1 \leq m \leq h+\alpha$, the box containing it in $\varphi_{\lambda,h,\alpha}\big(\phi_{\lambda,h,\alpha}(T)\big)$ is \[\begin{cases} v_{\phi_{\lambda,h,\alpha}(T)}(m) & 1 \leq m\leq\alpha-1, \\ v_{\phi_{\lambda,h,\alpha}(T)}(m+1) & \alpha \leq m \leq p-1+\alpha\mathrm{\ \ or\ \ }p+1+\alpha \leq m \leq h-1+\alpha, \\ v_{\phi_{\lambda,h,\alpha}(T)}(\alpha) & m=p+\alpha, \\ v_{\phi_{\lambda,h,\alpha}(T)}(p+1+\alpha) & m=h+\alpha. \end{cases}\] Here the inequalities $\alpha \leq m \leq p-1+\alpha$ or $p+1+\alpha \leq m \leq h-1+\alpha$ transform into $1+\alpha \leq m+1 \leq p+\alpha$ or $p+2+\alpha \leq m+1 \leq h+\alpha$, so that Definition \ref{mapincalpha} expresses this as \[\begin{cases} v_{T}(m) & 1 \leq m \leq \alpha-1, \\ v_{T}(m+1-1) & \alpha \leq m \leq p-1+\alpha\mathrm{\ \ or\ \ }p+1+\alpha \leq m \leq h-1+\alpha, \\ v_{T}(p+\alpha) & m=p+\alpha, \\ v_{T}(h+\alpha) & m=h+\alpha, \end{cases}\] which is again $v_{T}(m)$ for any $m$ and the asserted equality follows. This proves the proposition.
\end{proof}

Since Proposition \ref{invmaps} implies that maps from Definitions \ref{mapredalpha} and \ref{mapincalpha} are bijections between $\operatorname{SYT}_{h,\alpha}(\lambda)$ and $\operatorname{SYT}_{h,\alpha-1}(\lambda)$, we deduce the following immediate consequence about their cardinalities.
\begin{cor}
Given $h>0$ and $\alpha>0$, for any $\lambda \vdash h+\alpha$ we have $f_{h,\alpha}^{\lambda}=f_{h,\alpha-1}^{\lambda}$. \label{samesize}
\end{cor}

We can now prove our independence result.
\begin{thm}
Consider any number $k$, a partition $\lambda \vdash k$, and an integer $1 \leq h \leq k$. Then the size $f_{h,\alpha}^{\lambda}$ of the set $\operatorname{SYT}_{h,\alpha}(\lambda)$ from Definition \ref{SYThalpha} does not depend on the parameter $0\leq\alpha \leq k-h$.
\label{indepofalpha}
\end{thm}

\begin{proof}
It is suffice to consider $0<\alpha \leq k-h$, and show that $f_{h,\alpha}^{\lambda}=f_{h,\alpha-1}^{\lambda}$. We argue by induction on $k$, which we know that must satisfy $k \geq h+\alpha$. The case $k=h+\alpha$ is covered by Corollary \ref{indepofalpha}, so consider $\lambda \vdash k$, and assume that the equality holds for every partition of $k-1$. But then $0\leq\alpha<k-h$, so that we can invoke Lemma \ref{halphabr} and obtain
\[f_{h,\alpha}^{\lambda}=\sum_{v\in\operatorname{IC}(\lambda)}f_{h,\alpha}^{\lambda-v}=\sum_{v\in\operatorname{IC}(\lambda)}f_{h,\alpha-1}^{\lambda-v}=f_{h,\alpha-1}^{\lambda},\] where the middle equality follows from the induction hypothesis on $\lambda-v \vdash k-1$. This completes the proof of the theorem.
\end{proof}
It immediately follows from Theorems \ref{coeffSYT} and \ref{indepofalpha} that the coefficients $a_{\lambda,h}$ from Equation (\ref{expoff}) equal $f_{h,\alpha}^{\lambda}$ for every $0\leq\alpha \leq k-h$, when $\lambda \vdash k$ and $k \geq h$.

\begin{rem}
In fact, the constructions from Definitions \ref{mapredalpha} and \ref{mapincalpha} can be generalized to maps, still denoted by $\varphi_{\lambda,h,\alpha}$ and $\phi_{\lambda,h,\alpha}$, for every $k \geq h+\alpha$ and $\lambda \vdash k$. Explicitly, for a tableau $T\in\operatorname{SYT}_{h,\alpha}(\lambda)\cap\operatorname{SYT}_{h,\alpha-1}(\lambda)$ (for any such $k$ and $\lambda$) we set $\varphi_{\lambda,h,\alpha}(T):=\phi_{\lambda,h,\alpha}(T):=T$. When $T\in\operatorname{SYT}_{h,\alpha}(\lambda)\setminus\operatorname{SYT}_{h,\alpha-1}(\lambda)$, we define $q$ precisely as in Definition \ref{mapredalpha}, and $\varphi_{\lambda,h,\alpha}(T)$ is the tableau of shape $\lambda$ that is given, for $1 \leq m \leq k$, by
\begin{align*}
v_{\varphi_{\lambda,h,\alpha}(T)}(m) & =\begin{cases} v_{T}(m) & 1 \leq m\leq\alpha-1\mathrm{\ \ or\ \ }h+1+\alpha \leq m \leq k, \\ v_{T}(m+1) & \alpha \leq m \leq q-2+\alpha\mathrm{\ \ or\ \ }q+\alpha \leq m \leq h-1+\alpha, \\ v_{T}(\alpha) & m=q-1+\alpha, \\ v_{T}(q+\alpha) & m=h+\alpha. \end{cases}
\end{align*}
In case $T\in\operatorname{SYT}_{h,\alpha-1}(\lambda)\setminus\operatorname{SYT}_{h,\alpha}(\lambda)$, let $p$ be as in Definition \ref{mapincalpha}, and we define $\phi_{\lambda,h,\alpha}(T)$ to be the tableau of shape $\lambda$ which is defined, for $1 \leq m \leq k$, by
\begin{align*}
v_{\phi_{\lambda,h,\alpha}(T)}(m) & =\begin{cases} v_{T}(m) & 1 \leq m \leq \alpha-1\mathrm{\ \ or\ \ }h+1+\alpha \leq m \leq k,  \\ v_{T}(m-1) & 1+\alpha \leq m \leq p+\alpha\mathrm{\ \ or\ \ }p+2+\alpha \leq m \leq h+\alpha, \\ v_{T}(h+\alpha) & m=p+1+\alpha, \\ v_{T}(p+\alpha) & m=\alpha. \end{cases}
\end{align*}
Then $\varphi_{\lambda,h,\alpha}$ takes $\operatorname{SYT}_{h,\alpha}(\lambda)$ into $\operatorname{SYT}_{h,\alpha-1}(\lambda)$, $\phi_{\lambda,h,\alpha}$ sends $\operatorname{SYT}_{h,\alpha-1}(\lambda)$ into $\operatorname{SYT}_{h,\alpha}(\lambda)$, and they are inverse mappings. To see this, note that for general $k$ the maps can be obtained by taking out the boxes containing the numbers between $h+1+\alpha$ and $k$ (in decreasing order), applying the original maps from Definitions \ref{mapredalpha} and \ref{mapincalpha}, and then putting the removed boxes back again (in the opposite order), in correspondence with the application of Lemma \ref{halphabr} in the proof of Theorem \ref{indepofalpha}. The desired properties now follow directly from Propositions \ref{decalpha}, \ref{enlalpha}, and \ref{invmaps}, since Remark \ref{forBR} and the proof of Lemma \ref{halphabr} show that removing and adding back the boxes with larger numbers take standard tableaux to standard tableaux and does not affect the $(h,\alpha)$ and $(h,\alpha-1)$-conditions. \label{genmaps}
\end{rem}
One can, in fact, start from the maps from Remark \ref{genmaps} directly, and establish Theorem \ref{indepofalpha} using them. However, the proofs of their required properties are harder in this case (see, e.g., the use of an inner corner property in Lemma \ref{qalphaIC}, and for general $k \geq h+\alpha$ we no longer have an inner corner), so the easier way is to carry out the technicalities when $k=h+\alpha$ and work the general case from it.

\smallskip

We conclude by writing the maps from Remark \ref{genmaps} explicitly for $\lambda=(3,2,1)\vdash6$ and $h=2$. Then $f^{\lambda}=16$ and $\lambda=\lambda^{t}$, so the proof of Corollary \ref{symev} implies that $f_{h,\alpha}^{\lambda}=8$ for every value of $0\leq\alpha\leq4$. The maps from Remark \ref{genmaps} between all the possible tableaux are given in the table in the next page.

\pagebreak{}

\[\begin{array}{ccccccccc} \alpha=0 &  & \alpha=1 &  & \alpha=2 &  & \alpha=3 &  & \alpha=4 \\ \\ \begin{bmatrix}1 & 3 & 4 \\ 2 & 5 \\ 6\end{bmatrix} & \stackrel[\varphi_{\lambda,2,0}]{\phi_{\lambda,2,0}}{\rightleftharpoons} & \begin{bmatrix}1 & 2 & 4 \\ 3 & 5 \\ 6\end{bmatrix} & \stackrel[\varphi_{\lambda,2,1}]{\phi_{\lambda,2,1}}{\rightleftharpoons} & \begin{bmatrix}1 & 2 & 3 \\ 4 & 5 \\ 6\end{bmatrix} & \stackrel[\varphi_{\lambda,2,2}]{\phi_{\lambda,2,2}}{\rightleftharpoons} & \begin{bmatrix}1 & 2 & 4 \\ 3 & 5 \\ 6\end{bmatrix} & \stackrel[\varphi_{\lambda,2,3}]{\phi_{\lambda,2,3}}{\rightleftharpoons} & \begin{bmatrix}1 & 2 & 4 \\ 3 & 5 \\ 6\end{bmatrix} \\ \\ \begin{bmatrix}1 & 3 & 4 \\ 2 & 6 \\ 5\end{bmatrix} & \stackrel[\varphi_{\lambda,2,0}]{\phi_{\lambda,2,0}}{\rightleftharpoons} & \begin{bmatrix}1 & 2 & 4 \\ 3 & 6 \\ 5\end{bmatrix} & \stackrel[\varphi_{\lambda,2,1}]{\phi_{\lambda,2,1}}{\rightleftharpoons} & \begin{bmatrix}1 & 2 & 3 \\ 4 & 6 \\ 5\end{bmatrix} & \stackrel[\varphi_{\lambda,2,2}]{\phi_{\lambda,2,2}}{\rightleftharpoons} & \begin{bmatrix}1 & 2 & 3 \\ 4 & 6 \\ 5\end{bmatrix} & \stackrel[\varphi_{\lambda,2,3}]{\phi_{\lambda,2,3}}{\rightleftharpoons} & \begin{bmatrix}1 & 2 & 3 \\ 4 & 5 \\ 6\end{bmatrix} \\ \\ \begin{bmatrix}1 & 3 & 5 \\ 2 & 4 \\ 6\end{bmatrix} & \stackrel[\varphi_{\lambda,2,0}]{\phi_{\lambda,2,0}}{\rightleftharpoons} & \begin{bmatrix}1 & 2 & 5 \\ 3 & 4 \\ 6\end{bmatrix} & \stackrel[\varphi_{\lambda,2,1}]{\phi_{\lambda,2,1}}{\rightleftharpoons} & \begin{bmatrix}1 & 3 & 5 \\ 2 & 4 \\ 6\end{bmatrix} & \stackrel[\varphi_{\lambda,2,2}]{\phi_{\lambda,2,2}}{\rightleftharpoons} & \begin{bmatrix}1 & 3 & 4 \\ 2 & 5 \\ 6\end{bmatrix} & \stackrel[\varphi_{\lambda,2,3}]{\phi_{\lambda,2,3}}{\rightleftharpoons} & \begin{bmatrix}1 & 3 & 4 \\ 2 & 5 \\ 6\end{bmatrix} \\ \\ \begin{bmatrix}1 & 3 & 6 \\ 2 & 5 \\ 4\end{bmatrix} & \stackrel[\varphi_{\lambda,2,0}]{\phi_{\lambda,2,0}}{\rightleftharpoons} & \begin{bmatrix}1 & 2 & 6 \\ 3 & 5 \\ 4\end{bmatrix} & \stackrel[\varphi_{\lambda,2,1}]{\phi_{\lambda,2,1}}{\rightleftharpoons} & \begin{bmatrix}1 & 2 & 6 \\ 3 & 5 \\ 4\end{bmatrix} & \stackrel[\varphi_{\lambda,2,2}]{\phi_{\lambda,2,2}}{\rightleftharpoons} & \begin{bmatrix}1 & 2 & 6 \\ 3 & 4 \\ 5\end{bmatrix} & \stackrel[\varphi_{\lambda,2,3}]{\phi_{\lambda,2,3}}{\rightleftharpoons} & \begin{bmatrix}1 & 2 & 5 \\ 3 & 4 \\ 6\end{bmatrix} \\ \\ \begin{bmatrix}1 & 3 & 6 \\ 2 & 4 \\ 5\end{bmatrix} & \stackrel[\varphi_{\lambda,2,0}]{\phi_{\lambda,2,0}}{\rightleftharpoons} & \begin{bmatrix}1 & 2 & 6 \\ 3 & 4 \\ 5\end{bmatrix} & \stackrel[\varphi_{\lambda,2,1}]{\phi_{\lambda,2,1}}{\rightleftharpoons} & \begin{bmatrix}1 & 3 & 6 \\ 2 & 4 \\ 5\end{bmatrix} & \stackrel[\varphi_{\lambda,2,2}]{\phi_{\lambda,2,2}}{\rightleftharpoons} & \begin{bmatrix}1 & 3 & 6 \\ 2 & 4 \\ 5\end{bmatrix} & \stackrel[\varphi_{\lambda,2,3}]{\phi_{\lambda,2,3}}{\rightleftharpoons} & \begin{bmatrix}1 & 3 & 5 \\ 2 & 4 \\ 6\end{bmatrix} \\ \\ \begin{bmatrix}1 & 4 & 5 \\ 2 & 6 \\ 3\end{bmatrix} & \stackrel[\varphi_{\lambda,2,0}]{\phi_{\lambda,2,0}}{\rightleftharpoons} & \begin{bmatrix}1 & 4 & 5 \\ 2 & 6 \\ 3\end{bmatrix} & \stackrel[\varphi_{\lambda,2,1}]{\phi_{\lambda,2,1}}{\rightleftharpoons} & \begin{bmatrix}1 & 3 & 5 \\ 2 & 6 \\ 4\end{bmatrix} & \stackrel[\varphi_{\lambda,2,2}]{\phi_{\lambda,2,2}}{\rightleftharpoons} & \begin{bmatrix}1 & 3 & 4 \\ 2 & 6 \\ 5\end{bmatrix} & \stackrel[\varphi_{\lambda,2,3}]{\phi_{\lambda,2,3}}{\rightleftharpoons} & \begin{bmatrix}1 & 3 & 5 \\ 2 & 6 \\ 4\end{bmatrix} \\ \\ \begin{bmatrix}1 & 3 & 5 \\ 2 & 6 \\ 4\end{bmatrix} & \stackrel[\varphi_{\lambda,2,0}]{\phi_{\lambda,2,0}}{\rightleftharpoons} & \begin{bmatrix}1 & 2 & 5 \\ 3 & 6 \\ 4\end{bmatrix} & \stackrel[\varphi_{\lambda,2,1}]{\phi_{\lambda,2,1}}{\rightleftharpoons} & \begin{bmatrix}1 & 2 & 5 \\
3 & 6 \\ 4\end{bmatrix} & \stackrel[\varphi_{\lambda,2,2}]{\phi_{\lambda,2,2}}{\rightleftharpoons} & \begin{bmatrix}1 & 2 & 4 \\ 3 & 6 \\ 5\end{bmatrix} & \stackrel[\varphi_{\lambda,2,3}]{\phi_{\lambda,2,3}}{\rightleftharpoons} & \begin{bmatrix}1 & 2 & 5 \\ 3 & 6 \\ 4\end{bmatrix} \\
\\ \begin{bmatrix}1 & 4 & 6 \\ 2 & 5 \\ 3\end{bmatrix} & \stackrel[\varphi_{\lambda,2,0}]{\phi_{\lambda,2,0}}{\rightleftharpoons} & \begin{bmatrix}1 & 4 & 6 \\ 2 & 5 \\ 3\end{bmatrix} & \stackrel[\varphi_{\lambda,2,1}]{\phi_{\lambda,2,1}}{\rightleftharpoons} & \begin{bmatrix}1 & 3 & 6 \\ 2 & 5 \\ 4\end{bmatrix} & \stackrel[\varphi_{\lambda,2,2}]{\phi_{\lambda,2,2}}{\rightleftharpoons} & \begin{bmatrix}1 & 4 & 6 \\ 2 & 5 \\ 3\end{bmatrix} & \stackrel[\varphi_{\lambda,2,3}]{\phi_{\lambda,2,3}}{\rightleftharpoons} & \begin{bmatrix}1 & 4 & 5 \\ 2 & 6 \\ 3\end{bmatrix} \end{array}\]

\medskip

\noindent\textsc{Einstein Institute of Mathematics, the Hebrew University of Jerusalem, Edmund Safra Campus, Jerusalem 91904, Israel}

\noindent E-mail address: avichai.cohen2@mail.huji.ac.il

\noindent\textsc{Einstein Institute of Mathematics, the Hebrew University of Jerusalem, Edmund Safra Campus, Jerusalem 91904, Israel}

\noindent E-mail address: zemels@math.huji.ac.il

\end{document}